\documentclass[11pt,letterpaper,reqno]{amsart}
\usepackage{graphicx, psfrag, amscd, amssymb}

\newtheorem{lem}{Lemma}[section]
\newtheorem{thm}[lem]{Theorem}
\newtheorem{pro}[lem]{Proposition}
\newtheorem{cor}[lem]{Corollary}
\newtheorem{exa}[lem]{Example}

\newtheorem{defi}[lem]{Definition}

\newtheorem{rem}[lem]{Remarks}

\newcommand{\des}{{\mathsf{des}}}
\newcommand{\exc}{{\mathsf{exc}}}

\newcommand{\fwex}{{\mathsf{fwex}}}
\newcommand{\cro}{{\mathsf{cro}}}
\newcommand{\nega}{{\mathsf{neg}}}
\newcommand{\wex}{{\mathsf{wex}}}
\newcommand{\fexc}{{\mathsf{fexc}}}

\newcommand{\cs}{\mathsf{cs}}
\newcommand{\pat}{\mathsf{pat}}

\newcommand{\acb}{\mathsf{13\textnormal{-}2}}
\newcommand{\bca}{\mathsf{2\textnormal{-}31}}

\renewcommand{\SS}{{\mathcal{S}}}

\newcommand{\F}{\mathcal{F}}
\newcommand{\G}{\mathcal{G}}
\renewcommand{\H}{\mathcal{H}}
\newcommand{\M}{\mathcal{M}}

\newcommand{\T}{\mathcal{T}}
\newcommand{\U}{\mathsf{U}}
\renewcommand{\L}{\mathsf{L}}
\newcommand{\W}{\mathsf{W}}
\newcommand{\D}{\mathsf{D}}

\title[Signed Countings of types B and D permutations]{Signed Countings of types B and D permutations \\and $t,q$-Euler Numbers}

\author{Sen-Peng Eu}
\address{Department of Mathematics, National Taiwan Normal University, Taipei 11677, and Chinese Air Force Academy, Kaohsiung 82047, Taiwan, ROC}
\email{speu@math.ntnu.edu.tw}

\author{Tung-Shan Fu}
\address{Department of Applied Mathematics, National Pingtung University, Pingtung 90003, Taiwan, ROC}
\email{tsfu@mail.nptu.edu.tw}

\author{Hsiang-Chun Hsu}
\address{Department of Mathematics, National Taiwan Normal University, Taipei 11677, Taiwan, ROC}
\email{hchsu0222@gmail.com}

\author{Hsin-Chieh Liao}
\address{Department of Mathematics, National Taiwan Normal University, Taipei 11677, Taiwan, ROC}
\email{jeffliao1@gmail.com}



\begin{document}

\begin{abstract} It is a classical result that the parity-balance of the number of weak excedances of all permutations (derangements, respectively) of length $n$ is the Euler number $E_n$, alternating in sign, if $n$ is odd (even, respectively). Josuat-Verg\`{e}s obtained a $q$-analog of the results respecting the number of crossings of a permutation. One of the goals in this paper is to extend the results to the permutations (derangements, respectively) of types B and D, on the basis of the joint distribution in statistics excedances, crossings and the number of negative entries obtained by Corteel, Josuat-Verg\`{e}s and Kim.

Springer numbers are analogous Euler numbers that count the alternating
permutations of type B, called snakes. Josuat-Verg\`{e}s derived bivariate polynomials $Q_n(t,q)$ and $R_n(t,q)$ as generalized Euler numbers via
successive $q$-derivatives and multiplications by $t$ on polynomials in $t$. The other goal in this paper is to give a combinatorial interpretation of  $Q_n(t,q)$ and $R_n(t,q)$ as the enumerators of the snakes with restrictions.
\end{abstract}

\maketitle

\section{Introduction}
Let $\mathfrak{S}_n$ be the set of
permutations of $[n]:=\{1,2,\ldots,n\}$.
The classical Euler numbers $E_n$, defined by
\[
\sum_{n\ge 0}E_n\frac{x^n}{n!}=\tan x+\sec x,
\]
count the number of \emph{alternating permutations} in $\mathfrak{S}_n$, i.e., $\sigma\in\mathfrak{S}_n$ such that
$\sigma_1>\sigma_2<\sigma_3>\ldots \sigma_n$. The numbers $E_{2n}$ are called the \emph{secant numbers} and the numbers $E_{2n+1}$ are called the \emph{tangent numbers}.

With the sign of a permutation depending on the parity of the number of weak excedances, it is known \cite{Foata-Schu, Roselle1968} that the signed counting of all permutations (derangements, respectively) of $[n]$ is the Euler number $E_n$ if $n$ is odd (even, respectively). There are also $q$-analogs of the results, refined by various statistics \cite{FH,JVerges2010,Shin2010}. One of the goals in this paper is to give their types B and D extensions. Surprisingly, the signs are controlled by one half of the number of flag weak excedances (Theorems \ref{thm:signFwexB} and \ref{thm:signFwexD}).

As alternating permutations form the largest descent class in the symmetric group $\mathfrak{S}_n$, Springer \cite{Springer} determined the cardinality $K(W)$ of the largest descent class of any finite Coxeter group $W$ in the classification. Arnol'd \cite{Arnold} introduced families of restricted signed permutations, called \emph{snakes}, with quantities $K(W)$ for each group $W$. In particular, the number $S_n=K(B_n)$ associated with the group $B_n$ of signed permutations is called the $n$th \emph{Springer number}. Springer \cite{Springer} obtained the exponential generating function for $\{S_n\}_{n\ge 0}$,
\[
\sum_{n\ge 0} S_n\frac{x^n}{n!}=\frac{1}{\cos x-\sin x}.
\]
Josuat-Verg\`{e}s derived bivariate polynomials $Q_n(t,q)$ and $R_n(t,q)$ as generalized Euler numbers via
successive $q$-derivatives and multiplications by $t$ on polynomials in $t$. The other goal in this paper is to give a combinatorial interpretation of  $Q_n(t,q)$ and $R_n(t,q)$ as the enumerators of variants of the snakes.

\subsection{Backgrounds}
For a permutation
$\sigma=\sigma_1\sigma_2\cdots\sigma_n\in\mathfrak{S}_n$, an \emph{excedance} (\emph{weak excedance}, respectively) is an integer $i\in [n]$ such that $\sigma_i>i$ ($\sigma_i\ge i$, respectively). Let $\exc(\sigma)$ and $\wex(\sigma)$ denote the number of excedances and the number of weak excedances of $\sigma$, respectively.
It is a classical result of Euler that
\begin{equation}\label{Eulercan1}
    \sum_{\sigma\in\mathfrak{S}_n}(-1)^{\exc(\sigma)}=
    \left\{\begin{array}{ll}
        0  & \mbox{if $n$ is even,}\\
    (-1)^{\frac{n-1}{2}}E_n & \mbox{if $n$ is odd.}
    \end{array}
    \right.
\end{equation}
This is the evaluation at $x=-1$ of the $\gamma$-expansion of the Eulerian polynomials, obtained by Foata and Sch\"utzenberger~\cite{Foata-Schu},
\begin{equation}
 \sum_{\sigma\in\mathfrak{S}_n} x^{\exc(\sigma)}=\sum_{i=0}^{\lfloor (n-1)/2\rfloor} \gamma_{n,i}x^{i}(1+x)^{n-1-2i},
\end{equation}
where $\gamma_{n,i}$ is the number of $\sigma\in\mathfrak{S}_n$ with $i$ descents and no double descents.

Let $\mathfrak{S}_n^{*}$ be the set of \emph{derangements} in $\mathfrak{S}_n$ (i.e., consisting of $\sigma\in\mathfrak{S}_n$ without fixed points). Another result is that
\begin{equation}\label{Eulercan2}
    \sum_{\sigma\in\mathfrak{S}_n^{*}}(-1)^{\exc(\sigma)}=
    \left\{
            \begin{array}{ll}
        (-1)^{\frac{n}{2}}E_n  & \mbox{if $n$ is even,}\\
        0 & \mbox{if $n$ is odd.}
            \end{array}
    \right.
\end{equation}
This is the evaluation at $x=-1$ of the $\gamma$-expansion of the derangement polynomials
\begin{equation}
\sum_{\sigma\in\mathfrak{S}_n^{*}}x^{\exc(\sigma)}=\sum_{i=0}^{\lfloor n/2\rfloor} \xi_{n,i}x^{i}(1+x)^{n-2i},
\end{equation}
where $\xi_{n,i}$ is the number of $\sigma\in\mathfrak{S}_n$ with  $i$ decreasing runs and none of size one. This result along with $q$-analogs were obtained independently by several authors; see e.g. \cite{AS, LZ, SW, Shin2012}. Roselle~\cite{Roselle1968} obtained an analogous result of Eq.\,(\ref{Eulercan2}), using a different statistic with equidistribution, with the exponential generating function $\frac{2e^x}{e^{2x}+1}$ for the right-hand side.

There are some $q$-analogs of the Euler numbers $E_n$. One of the $q$-analogs is introduced by Han, Randrianarivony and Zeng \cite{HRZ}, defined in the form of continued fraction expansion.  We use the notation $\frac{\alpha_1}{\beta_1}{{}\atop{-}}\frac{\alpha_2}{\beta_2}{{}\atop{-}}\cdots=\alpha_1/(\beta_1-\alpha_2/(\beta_2-\dots))$ for continued fractions.
The $q$-\emph{secant numbers} $E_{2n}(q)$ and the $q$-\emph{tangent numbers} $E_{2n+1}(q)$ are defined by
\begin{align} \label{eqn:q-secant}
\sum_{n\ge 0} E_{2n}(q)x^n &=\frac{1}{1}{{}\atop{-}}\frac{[1]_q^2x}{1}{{}\atop{-}}\frac{[2]_q^2x}{1}{{}\atop{-}}\frac{[3]_q^2x}{1}{{}\atop{-}}\cdots, \\ \label{eqn:q-tangent}
\sum_{n\ge 0} E_{2n+1}(q)x^n &=\frac{1}{1}{{}\atop{-}}\frac{[1]_q[2]_qx}{1}{{}\atop{-}}\frac{[2]_q[3]_qx}{1}{{}\atop{-}}\frac{[3]_q[4]_qx}{1}{{}\atop{-}}\cdots,
\end{align}
where $[j]_q:=1+q+\dots +q^{j-1}$ for all positive integers $j$.

A \emph{crossing} of $\sigma\in\mathfrak{S}_n$ is a pair $(i,j)$ with $i,j\ge 1$ such that $i<j\le\sigma_i<\sigma_j$ or $\sigma_i<\sigma_j<i<j$, denoted by $\cro(\sigma)$ the number of crossings of $\sigma$. The notions of crossing and alignment of a permutation were introduced by Williams \cite{Williams} in connection to totally positive Grassmann cells. Corteel \cite{Corteel-cross} gave an enumerative relation between crossing and alignment. Refined by the statistic $\cro(\sigma)$, Josuat-Verg\`{e}s~\cite{JVerges2010} derived $q$-analogs of the results in Eqs.\,(\ref{Eulercan1}) and (\ref{Eulercan2}),
\begin{align}
 \sum_{\sigma\in\mathfrak{S}_n}(-1)^{\wex(\sigma)}q^{\cro(\sigma)}
    &=\left\{\begin{array}{ll}
            0  & \mbox{if $n$ is even,}\\
            (-1)^{\frac{n+1}{2}}E_n(q) & \mbox{if $n$ is odd;}
           \end{array}\right.   \label{JV1} \\
 \sum_{\sigma\in\mathfrak{S}_n^{*}} \left(-\frac{1}{q}\right)^{\wex(\sigma)}q^{\cro(\sigma)}
    &=\left\{\begin{array}{ll}
                        \big(-\frac{1}{q}\big)^{\frac{n}{2}}E_n(q)  & \mbox{if $n$ is even,}\\
                        0 & \mbox{if $n$ is odd.}
                    \end{array}\right.  \label{JV2}
\end{align}
There are two other distinct $q$-analogs of Eqs.\,(\ref{Eulercan1}) and (\ref{Eulercan2}) that have been discovered by Foata and Han \cite{FH} and by Shin and Zeng \cite{Shin2010}, respectively. One of the purposes of this paper is to extend the results in Eqs.\,(\ref{JV1}) and (\ref{JV2}) on signed permutations and even-signed permutations.

\subsection{Main results}
Josuat-Verg\`{e}s introduced bivariate polynomials $Q_n(t,q)$ and $R_n(t,q)$ for generalizations of the $q$-secant numbers and $q$-tangent numbers, respectively \cite{JV}. Let $D$ be the $q$-analog of the differential operator acting on polynomials $f(t)$ by
\begin{equation}
(D f)(t) := \frac{f(qt)-f(t)}{(q-1)t}.
\end{equation}
Let $U$ be the operator acting on $f(t)$ by multiplication by $t$. Notice that the $q$-derivative $D(t^n)=[n]_q t^{n-1}$ and the communication relation $DU-qUD=1$ hold. The polynomials $Q_n(t,q)$ and $R_n(t,q)$ are defined algebraically by
\begin{equation}
Q_n(t,q):=(D+UDU)^n 1, \qquad R_n(t,q):= (D+DUU)^n 1.
\end{equation}
Several of the initial polynomials are listed below:
\begin{align*}
Q_0(t,q) &= 1 \\
Q_1(t,q) &= t \\
Q_2(t,q) &= 1+(1+q)t^2 \\
Q_3(t,q) &= (2+2q+q^2)t+(1+2q+2q^2+q^3)t^3, \\
& \\
R_0(t,q) &= 1 \\
R_1(t,q) &= (1+q)t \\
R_2(t,q) &= (1+q)+(1+2q+2q^2+q^3)t^2 \\
R_3(t,q) &= (2+5q+5q^2+3q^3+q^4)t+(1+3q+5q^2+6q^3+5q^4+3q^5+q^6)t^3.
\end{align*}
Josuat-Verg\`{e}s proved $Q_n(1,1)=S_n$ and $R_n(1,1)=2^nE_{n+1}$, and obtained the generating functions for the sequences $\{Q_n(t,q)\}_{n\ge 0}$ and $\{R_n(t,q)\}_{n\ge 0}$ in the form of J-fractions (or Jacobi continued fractions) \cite{JV}.

\begin{defi} \label{def:J-fraction} For any two sequences $\{\mu_h\}_{h\ge 0}$ and $\{\lambda_h\}_{h\ge 1}$, let $\mathfrak{F}(\mu_h,\lambda_h)$ denote the continued fraction
\begin{equation}
\mathfrak{F}(\mu_h,\lambda_h)=\frac{1}{1-\mu_0 x}{{}\atop{-}}\frac{\lambda_1 x^2}{1-\mu_1 x}{{}\atop{-}}\frac{\lambda_2 x^2}{1-\mu_2 x}{{}\atop{-}}\frac{\lambda_3 x^2}{1-\mu_3x}{{}\atop{-}}\cdots
\end{equation}
\end{defi}

\begin{thm} {\rm\bf (Josuat-Verg\`es)} \label{thm:QR-continued-fraction} We have
\begin{equation}
\sum_{n\ge 0} Q_n(t,q) x^n = \mathfrak{F}(\mu^{Q}_h,\lambda^{Q}_h), \qquad
\sum_{n\ge 0} R_n(t,q) x^n = \mathfrak{F}(\mu^{R}_h,\lambda^{R}_h),
\end{equation}
where
\[
\left\{\begin{array}{rl}
\mu^{Q}_h &= tq^h([h]_q+[h+1]_q)\\ [0.8ex]
\lambda^{Q}_h &= (1+t^2q^{2h-1})[h]^2_q,
       \end{array}
\right.
\qquad
\left\{\begin{array}{rl}
\mu^{R}_h &= tq^h(1+q)[h+1]_q\\ [0.8ex]
\lambda^{R}_h &= (1+t^2q^{2h})[h]_q[h+1]_q.
       \end{array}
\right.
\]
\end{thm}
Notice that $Q_{2n}(0,q)=E_{2n}(q)$ and $R_{2n+1}(0,q)=E_{2n+1}(q)$, the $q$-secant and $q$-tangent numbers defined in Eqs.\,(\ref{eqn:q-secant}) and (\ref{eqn:q-tangent}).

A \emph{signed permutation} of $[n]$ is a bijection $\sigma$ of the set $[\pm n]:=\{-n,-n+1,\dots,-1,1,2,\dots,n\}$ onto itself such that $\sigma(-i)=-\sigma(i)$ for all $i\in[\pm n]$. It will be denoted as $\sigma=\sigma_1\sigma_2\cdots\sigma_n=(\sigma_1,\sigma_2,\dots,\sigma_n)$, where $\sigma_i=\sigma(i)$, called the \emph{window notation} of $\sigma$. Let $B_n$ be the set of signed permutations of $[n]$. Let $D_n\subset B_n$ be the subset consisting of the signed permutations with even number of negative entries in their window notations. Members of $D_n$ are called the \emph{even-signed permutations} of $[n]$.

For a $\sigma=\sigma_1\sigma_2\cdots\sigma_n\in B_n$, a \emph{crossing} of $\sigma$ is a pair $(i,j)$  with $i,j\ge 1$ such that
\begin{itemize}
  \item $i<j\le \sigma_i< \sigma_j$ or
  \item $-i<j\le -\sigma_i < \sigma_j$ or
  \item $i>j>\sigma_i>\sigma_j$.
\end{itemize}
The notion of crossing of signed permutations was first considered in \cite{CJW}. Let $\cro_B(\sigma)$ be the number of crossings of $\sigma$. For example, the crossings of $\sigma=(3,-4,-2,5,1)\in B_5$ consist of the pairs $(i,j)\in\{(3,1), (2,4)\}$ ($\{(3,2), (5,2), (5,3)\}$ respectively), which satisfy the second (third, respectively) condition. Hence $\cro_B(\sigma)=5$.

Let $\nega(\sigma)$ be the number of negative entries in $\sigma$ and let $\fwex(\sigma)=2\wex(\sigma)+\nega(\sigma)$, where $\wex(\sigma)=\#\{i\in [n]: \sigma_i\ge i\}$. In particular, $\sigma_i$ is a \emph{fixed point} of $\sigma$ if $\sigma_i=i$.
Let $B_n^{*}$ ($D_n^{*}$, respectively) denote the subset of $B_n$ ($D_n$, respectively) consisting of the signed permutations of $[n]$ without fixed points.

Our first main result is the type B and type D analogs of
the results in Eqs.\,(\ref{JV1}) and (\ref{JV2}), with the sign of $\sigma\in B_n$ depending on the parity of one half of the statistic $\fwex(\sigma)$.

\begin{thm} \label{thm:signFwexB}
For $n\ge 1$, we have
\begin{enumerate}
\item  ${\displaystyle
\sum_{\sigma\in B_n}(-1)^{\lfloor \frac{\fwex(\sigma)}{2}\rfloor} t^{\nega(\sigma)}q^{\cro_B(\sigma)}
=\left[(-1)^{\lfloor\frac{n+1}{2}\rfloor}+(-1)^{\lfloor\frac{n}{2}\rfloor} t \right] R_{n-1}(t,q).
}$
\item  ${\displaystyle
\sum_{\sigma\in D_n}(-1)^{\lfloor\frac{\fwex(\sigma)}{2}\rfloor} t^{\nega(\sigma)}q^{\cro_B(\sigma)}
     =\left\{\begin{array}{ll} (-1)^\frac{n}{2}tR_{n-1}(t,q)  & \mbox{ if $n$ is even,}\\
	                           (-1)^\frac{n+1}{2}R_{n-1}(t,q) & \mbox{ if $n$ is odd.}
             \end{array}
      \right.
}$
\end{enumerate}
\end{thm}

\begin{thm} \label{thm:signFwexD}
For $n\ge 1$, we have
\begin{enumerate}
\item ${\displaystyle
\sum_{\sigma\in B_n^{*}} \left(-\frac{1}{q}\right)^{\lfloor\frac{\fwex(\sigma)}{2}\rfloor} t^{\nega(\sigma)}q^{\cro_B(\sigma)}
=\left(-\frac{1}{q}\right)^{\lfloor\frac{n}{2}\rfloor}Q_n(t,q).
}$
\item ${\displaystyle
\sum_{\sigma\in D_n^{*}} \left(-\frac{1}{q}\right)^{\lfloor\frac{\fwex(\sigma)}{2}\rfloor} t^{\nega(\sigma)}q^{\cro_B(\sigma)}
     =\left\{\begin{array}{ll} \big(-\frac{1}{q}\big)^\frac{n}{2}Q_{n}(t,q) & \mbox{ if $n$ is even,}\\
	                            0 & \mbox{ if $n$ is odd.}
             \end{array}
      \right.
}$
\end{enumerate}
\end{thm}
When $t=1$ and $q=1$, we obtain types B and D extensions of the results in Eqs.\,(\ref{Eulercan1}) and (\ref{Eulercan2}).

\begin{cor} \label{cor:(-1)-eval-Bn}
For $n\ge 1$, we have
\begin{enumerate}
\item ${\displaystyle
\sum_{\sigma\in B_n}(-1)^{\lfloor\frac{\fwex(\sigma)}{2}\rfloor}
      =\left\{\begin{array}{ll} (-1)^\frac{n}{2}2^nE_n & \mbox{ if $n$ is even,}\\
	                             0 &  \mbox{ if $n$ is odd.}
              \end{array}
       \right.
}$
\item ${\displaystyle
\sum_{\sigma\in D_n}(-1)^{\lfloor\frac{\fwex(\sigma)}{2}\rfloor}=(-1)^{\lfloor\frac{n+1}{2}\rfloor}2^{n-1} E_n.
}$
\end{enumerate}
\end{cor}

\begin{cor} \label{cor:(-1)-eval-Bn*}
For $n\ge 1$, we have
\begin{enumerate}
\item ${\displaystyle
\sum_{\sigma\in B_n^{*}}(-1)^{\lfloor\frac{\fwex(\sigma)}{2}\rfloor}=(-1)^{\lfloor\frac{n}{2}\rfloor} S_n.
}$
\item ${\displaystyle
\sum_{\sigma\in D_n^{*}}(-1)^{\lfloor\frac{\fwex(\sigma)}{2}\rfloor}
      =\left\{\begin{array}{ll} (-1)^\frac{n}{2} S_n & \mbox{ if $n$ is even,}\\
	                            0 & \mbox{ if $n$ is odd.}
              \end{array}
      \right.
}$
\end{enumerate}
\end{cor}

\begin{rem} {\rm
Corteel et al. \cite[Proposition 3.3]{CJK} proved that the number of type B descents $\des_B(\sigma)=\#\{i:\sigma_i>\sigma_{i+1}, 0\le i< n\}$ is equidistributed with $\lfloor \fwex(\sigma)/2\rfloor$ for $\sigma\in B_n$. The result in (i) of Corollary \ref{cor:(-1)-eval-Bn} coincides with the evaluation at $x=-1$ of the $\gamma$-nonnegativity result in \cite[Theorem 13.5]{Petersen}. See \cite[Theorem 1.2]{A} for a $\gamma$-nonnegativity result for the derangements $\sigma\in B_n^{*}$, making use of the statistic $\fexc(\sigma)=2\exc(\sigma)+\nega(\sigma)$. Since $\fwex(\sigma)$ agrees with $\fexc(\sigma)$ when $\sigma$ is a derangement, Athanasiadis' result  implies (i) of Corollary \ref{cor:(-1)-eval-Bn*} when setting $r=2$ and $x=-1$.
There is also a type D $\gamma$-nonnegativity result in \cite[Theorem 13.9]{Petersen}, which uses a varied descent number statistic of signed permutations. See \cite{Athana} for a recent survey of $\gamma$-positivity results appearing in various combinatorial or geometric contexts.
}
\end{rem}

\subsection{The generalized $t,q$-Euler numbers}
A signed permutation $\sigma=\sigma_1\sigma_2\cdots\sigma_n\in B_n$ is a \emph{snake} of size $n$ if $\sigma_1>\sigma_2<\sigma_3>\cdots \sigma_n$. Let $\SS_n\subset B_n$ be the set of snakes of size $n$. Let $\SS_n^{0}\subset \SS_n$ be the subset consisting of the snakes $\sigma$ with $\sigma_1>0$, and let $\SS_n^{00}\subset\SS_n^{0}$ be the subset consisting of the snakes $\sigma$ with $\sigma_1>0$ and $(-1)^n\sigma_n<0$. The sets $\SS_n$ and $\SS_n^{0}$ were introduced by Arnol'd \cite{Arnold} and the set $\SS_n^{00}$ was introduced by Josuat-Verg\`{e}s \cite{JV}.

For a snake $\sigma=\sigma_1\sigma_2\cdots\sigma_n\in B_n$, let $\cs(\sigma)$ denote the number of changes of sign through the entries $\sigma_1,\sigma_2,\dots,\sigma_n$, i.e., $\cs(\sigma):=\#\{i: \sigma_i\sigma_{i+1}<0,  0\le i\le n\}$ with the following convention for the entries $\sigma_0$ and $\sigma_{n+1}$:
\begin{itemize}
  \item $\sigma_0=-(n+1)$ and $\sigma_{n+1}=(-1)^n (n+1)$ if $\sigma\in\SS_n$;
  \item $\sigma_0=0$ and $\sigma_{n+1}=(-1)^n (n+1)$ if $\sigma\in\SS_n^{0}$;
  \item $\sigma_0=0$ and $\sigma_{n+1}=0$ if $\sigma\in\SS_n^{00}$.
\end{itemize}
Making use of a recursive method, Josuat-Verg\`{e}s \cite[Theorem 3.4]{JV} gave a combinatorial interpretation of $Q_n(t,1)$ and $R_n(t,1)$, with the distribution of $\cs(\sigma)$ for the variate $t$, i.e.,
\[
Q_n(t,1)=\sum_{\sigma\in\SS_n^{0}} t^{\cs(\sigma)}, \qquad R_n(t,1)=\sum_{\sigma\in\SS_{n+1}^{00}} t^{\cs(\sigma)}.
\]
Our second main result is a combinatorial interpretation of $Q_n(t,q)$ and $R_n(t,q)$. In terms of three-term patterns of permutations, we come up with the statistics $\pat_Q(\sigma)$ and $\pat_R(\sigma)$ of snakes (defined in Eq.\,(\ref{eqn:patQ}) and Eq.\,(\ref{eqn:patR}) in Section 5) for the variate $q$.

\begin{thm}
We have
\begin{enumerate}
\item ${\displaystyle
Q_n(t,q)=\sum_{\sigma\in \SS_n^{0}} t^{\cs(\sigma)}q^{\bca(|\sigma|)+\pat_Q(\sigma)}.
}$
\item ${\displaystyle
R_n(t,q)=\sum_{\sigma\in \SS^{00}_{n+1}} t^{\cs(\sigma)}q^{\bca(|\sigma|)+\pat_R(\sigma)-n-1}.
}$
\end{enumerate}
\end{thm}

\smallskip
The rest of the paper is organized as follows. In Section 2, we relate the polynomials $Q_n(t,q)$ and $R_n(t,q)$ to weighted bicolored Motzkin paths via the continued fraction expansions of their generating functions. We also review the connection between signed permutations and Motzkin paths established by Corteel et al. \cite{CJK}. In Section 3, we prove Theorem \ref{thm:signFwexB}, using the method of sign-reversing involution on the weighted paths. In Section 4, we consider the signed permutations without fixed points and prove Theorem \ref{thm:signFwexD}. In Section 5, we give a combinatorial interpretation of the polynomials $Q_n(t,q)$ and $R_n(t,q)$ as the enumerators of the snakes in $\SS^{0}_n$ and $\SS^{00}_{n+1}$, respectively.

\section{Continued fractions and weighted Motzkin paths}

A \emph{Motzkin path} of length $n$ is a lattice path from the origin to the point $(n,0)$ staying weakly above the $x$-axis, using the \emph{up step} $(1,1)$, \emph{down step} $(1,-1)$, and \emph{level step} $(1,0)$.  Let $\U$, $\D$ and $\L$ denote an up step, a down step and a level step, accordingly.

We consider a Motzkin path $\mu=z_1z_2\cdots z_n$ with a weight function $\rho$ on the steps. The \emph{weight} of $\mu$, denoted by $\rho(\mu)$, is defined to be the product of the weight $\rho(z_j)$ of each step $z_j$ for $j=1,2,\dots,n$.
The \emph{height} of a step $z_j$ is the $y$-coordinate of the starting point of $z_j$.
Making use of Flajolet's formula \cite[Proposition 7A]{Flaj}, the generating function for the weight count of the Motzkin paths can be expressed as a continued fraction.

\begin{thm} {\rm\bf (Flajolet)} \label{thm:Flajolet} For $h\ge 0$, let $a_h$, $b_h$ and $c_h$ be polynomials such that each monomial has coefficient 1. Let $M_n$ be the set of weighted Motzkin paths of length $n$ such that the weight of an up step  (down step or level step, respectively) at height $h$ is one of the monomials appearing in $a_h$ ($b_h$ or $c_h$, respectively). Then the generating function for $\rho(M_n)=\sum_{\mu\in M_n} \rho(\mu)$ has the expansion
\begin{equation} \label{eqn:cf-Motzkin}
\sum_{n\ge 0} \rho(M_n)x^n
      =\frac{1}{1-c_0x}{{}\atop{-}}\frac{a_0b_1x^2}{1-c_1x}{{}\atop{-}}\frac{a_1b_2x^2}{1-c_2x}{{}\atop{-}}\cdots
\end{equation}
\end{thm}

\subsection{Bicolored Motzkin paths} A \emph{bicolored Motzkin path} (also known as \emph{2-Motzkin path}) is a Motzkin path with two kinds of level steps, say \emph{straight} and \emph{wavy}, denoted by $\L$ and $\W$, respectively. For a nonnegative integer $h$, let $z^{(h)}$ denote a step $z$ at height $h$ in a bicolored Motzkin path for $z\in\{\U,\L,\W,\D\}$.

By Theorem \ref{thm:QR-continued-fraction}, the initial part of the expansion of the generating functions for $R_n(t,q)$ and $Q_n(t,q)$ are shown below.
\begin{align*}
\sum_{n\ge 0} R_n(t,q)x^n &=\frac{1}{1-t(1+q)[1]_qx}{{}\atop{-}}\frac{(1+t^2q^2)[1]_q[2]_qx^2}{1-tq(1+q)[2]_qx}{{}\atop{-}}\frac{(1+t^2q^4)[2]_q[3]_qx^2}{1-tq^2(1+q)[3]_qx}\cdots, \\
\sum_{n\ge 0} Q_n(t,q)x^n &=\frac{1}{1-t[1]_qx}{{}\atop{-}}\frac{(1+t^2q)[1]_q^2x^2}{1-tq([1]_q+[2]_q)x}{{}\atop{-}}\frac{(1+t^2q^3)[2]_q^2x^2}{1-tq^2([2]_q+[3]_q)x}\cdots.
\end{align*}

By Theorem \ref{thm:Flajolet}, the following observations provide feasible weighting schemes that realize the polynomials  $R_n(t,q)$ and $Q_n(t,q)$ in terms of the bicolored Motzkin paths. Those paths will be used to encode the snakes in $\SS^{00}_{n+1}$ and $\SS^{0}_n$ (Theorem \ref{thm:T->snakes-00} and Theorem \ref{thm:T*->snake-0})

\begin{pro} \label{pro:Tn-weighting}
Let $\T_n$ be the set of weighted bicolored Motzkin paths of length $n$ with a weight function $\rho$ such that for $h\ge 0$,
\begin{itemize}
\item $\rho(\U^{(h)})\in\{1,q,\dots,q^{h}\}\cup\{t^2q^{2h+2},t^2q^{2h+3},\dots,t^2q^{3h+2}\}$,
\item $\rho(\L^{(h)})\in\{tq^{h+1},tq^{h+2},\dots,tq^{2h+1}\}$,
\item $\rho(\W^{(h)})\in\{tq^{h},tq^{h+1},\dots,tq^{2h}\}$,
\item $\rho(\D^{(h+1)})\in\{1,q,\dots,q^{h+1}\}$.
\end{itemize}
Then we have
\[
\sum_{n\ge 0} \rho(\T_n)x^n = \sum_{n\ge 0} R_n(t,q)x^n.
\]
\end{pro}

\begin{pro} \label{pro:T*n-weighting}
Let $\T^*_n$ be the set of  weighted bicolored Motzkin paths of length $n$ containing no wavy level steps on the $x$-axis, with a weight function $\rho$ such that for $h\ge 0$,
\begin{itemize}
\item $\rho(\U^{(h)})\in\{1,q,\dots,q^{h}\}\cup\{t^2q^{2h+1},t^2q^{2h+2},\dots,t^2q^{3h+1}\}$,
\item $\rho(\L^{(h)})\in\{tq^{h},tq^{h+1},\dots,tq^{2h}\}$,
\item $\rho(\W^{(h)})\in\{tq^{h},tq^{h+1},\dots,tq^{2h-1}\}$ for $h\ge 1$,
\item $\rho(\D^{(h+1)})\in\{1,q,\dots,q^{h}\}$.
\end{itemize}
Then we have
\[
\sum_{n\ge 0} \rho(\T^{*}_n)x^n = \sum_{n\ge 0} Q_n(t,q)x^n.
\]
\end{pro}

\subsection{Linking signed permutations to bicolored Motzkin paths} Corteel et al. \cite{CJK} considered a refined enumerator for the signed permutations in $B_n$, with respect to the statistics $\fwex$, $\nega$ and $\cro$, and define a sequence of polynomials $B_n(y,t,q)$  for $n\ge 0$ by
\begin{equation}
B_n(y,t,q)=\sum_{\sigma\in B_n} y^{\fwex(\sigma)}t^{\nega(\sigma)}q^{\cro_B(\sigma)}.
\end{equation}
For example, we have
\begin{align*}
B_0(y,t,q) &= 1,\\
B_1(y,t,q) &= y^2+yt, \\
B_2(y,t,q) &= y^4+(2t+tq)y^3+(t^2q+t^2+1)y^2+ty.
\end{align*}
Extended from Foata--Zeilberger bijection \cite{Foata-Zeil}, Corteel et al. established a bijection between the set $B_n$ of signed permutations and the following set of weighted bicolored Motzkin paths \cite[Subsection 7.1]{CJK} and obtained a continued fraction expansion of the generating function for the enumerators $B_n(y,t,q)$  \cite[Theorem 6.1]{CJK}.

\begin{defi} {\rm
Let $\M_n$ be the set of weighted bicolored Motzkin paths of length $n$ containing no wavy level steps on the $x$-axis, with a weight function $\rho$ such that for $h\ge 0$,
\begin{itemize}
\item $\rho(\U^{(h)})\in\{y^2,y^2q,\dots,y^2q^{h}\}\cup\{ytq^{h},ytq^{h+1},\dots,ytq^{2h}\}$,
\item $\rho(\L^{(h)})\in\{y^2,y^2q,\dots,y^2q^{h}\}\cup\{ytq^{h},ytq^{h+1},\dots,ytq^{2h}\}$,
\item $\rho(\W^{(h)})\in\{1,q,\dots,q^{h-1}\}\cup\{ytq^{h},ytq^{h+1},\dots,ytq^{2h-1}\}$ for $h\ge 1$,
\item $\rho(\D^{(h+1)})\in\{1,q,\dots,q^{h}\}\cup\{ytq^{h+1},ytq^{h+2},\dots,ytq^{2h+1}\}$.
\end{itemize}
}
\end{defi}

\begin{rem} {\rm
Rather than the original weight function given in \cite{CJK}, we have interchanged the possible weights of the up steps and the down steps, in the sense of traversing the paths backward. This unifies the possible weights for the initial step (either $\U^{(0)}$ or $\L^{(0)}$), for the purpose of restructuring the weighted bicolored Motzkin paths (Proposition \ref{pro:restructure}).
}
\end{rem}

\begin{thm}  \label{thm:Corteel} {\rm\bf (Corteel, Josuat-Verg\`es, Kim)}  There is a bijection $\Gamma$ between $B_n$ and $\M_n$ such that
\begin{equation}
\sum_{\sigma\in B_n} y^{\fwex(\sigma)}t^{\nega(\sigma)}q^{\cro_B(\sigma)} =\rho(\M_n).
\end{equation}
Moreover, let $\mu_h=y^2[h+1]_q+[h]_q+ytq^{h}([h]_q+[h+1]_q)$ for $h\ge 0$ and $\lambda_h=[h]_q^2(y^2+ytq^{h-1})(1+ytq^h)$ for $h\ge 1$, then we have
\begin{equation}  \label{eqn:B-enumerator-cf}
\sum_{n\ge 0} B_n(y,t,q) x^n = \mathfrak{F}(\mu_h,\lambda_h),
\end{equation}
where $\mathfrak{F}(\mu_h,\lambda_h)$ is the expansion of the J-fraction given in Definition \ref{def:J-fraction}.
\end{thm}

\section{Proof of Theorem \ref{thm:signFwexB}}
In this section we present a combinatorial proof of Theorem \ref{thm:signFwexB}, making use of the method of sign-reversing involution on paths.
First, we restructure the weighted bicolored Motzkin paths in $\M_n$.

\begin{defi} {\rm
	Let $\H_n$ be the set of weighted bicolored Motzkin paths of length $n$ with a weight function $\rho$ such that for $h\ge 0$,
\begin{itemize}
\item $\rho(\U^{(h)})\in\{y^2,y^2q,\dots,y^2q^{h+1}\}\cup\{ytq^{h+1},ytq^{h+2},\dots,ytq^{2h+2}\}$,
\item $\rho(\L^{(h)})\in\{1,q,\dots,q^{h}\}\cup\{ytq^{h+1},ytq^{h+2},\dots,ytq^{2h+1}\}$,
\item $\rho(\W^{(h)})\in\{y^2,y^2q,\dots,y^2q^{h}\}\cup\{ytq^{h},ytq^{h+1},\dots,ytq^{2h}\}$,
\item $\rho(\D^{(h+1)})\in\{1,q,\dots,q^{h}\}\cup\{ytq^{h+1},ytq^{h+2},\dots,ytq^{2h+1}\}$.
\end{itemize}
}
\end{defi}	
	
\begin{pro} \label{pro:restructure}
	There is a two-to-one bijection $\Phi$ between $\M_n$ and $\H_{n-1}$ such that
	\[
	\rho(\M_n)=(y^2+yt)\rho(\H_{n-1}).
	\]
\end{pro}

\begin{proof} Given a path $\mu=p_1p_2\cdots p_n\in\M_n$, we create a weight-preserving path $z_1z_2\cdots z_{2n}$ of length $2n$ from $\mu$ as the intermediate stage, where $z_{2i-1}z_{2i}$ is determined from $p_i$ by
\[
z_{2i-1}z_{2i}=\left\{\begin{array}{ll}
                   \U\U & \mbox{ if $p_i=\U$,}\\
	               \U\D & \mbox{ if $p_i=\L$,}\\
                   \D\U & \mbox{ if $p_i=\W$,}\\
                   \D\D & \mbox{ if $p_i=\D$,}
              \end{array}
      \right.
\]
with weight $\rho(z_{2i-1})=\rho(p_i)$ and $\rho(z_{2i})=1$ for $1\le i\le n$. Note that $\rho(z_1)=\rho(p_1)\in\{y^2,yt\}$ since $p_1=\U^{(0)}$ or $\L^{(0)}$. Then we construct the corresponding path $\Phi(\mu)=p'_1p'_2\cdots p'_{n-1}$ from $z_2\cdots z_{2n-1}$ (with $z_1$ and $z_{2n}$ excluded), where $p'_j$ is determined by $z_{2j}z_{2j+1}$ with weight $\rho(p'_j)=\rho(z_{2j})\rho(z_{2j+1})$  according to the following cases
\[
p'_j=\left\{\begin{array}{ll}
                   \U & \mbox{ if $z_{2j}z_{2j+1}=\U\U$,}\\
	               \L & \mbox{ if $z_{2j}z_{2j+1}=\U\D$,}\\
                   \W & \mbox{ if $z_{2j}z_{2j+1}=\D\U$,}\\
                   \D & \mbox{ if $z_{2j}z_{2j+1}=\D\D$,}
              \end{array}
      \right.
\]
for $1\le j\le n-1$. Notice that $\rho(\mu)=\rho(p_1)\rho(\Phi(\mu))$ since $\rho(p'_j)=\rho(p_{j+1})$ for $1\le j\le n-1$. Hence $\rho(\mu)=y^2\rho(\Phi(\mu))$ or $yt\rho(\Phi(\mu))$. Moreover, the possible weights of $p'_j$ can be determined from the steps of $\mu$ since
\[
	p'_j=\left\{\begin{array}{ll}
	\U^{(h)} & \mbox{ if $p_{j+1}=\U^{(h+1)}$ or $\L^{(h+1)}$,}\\
	\L^{(h)} & \mbox{ if $p_{j+1}=\W^{(h+1)}$ or $\D^{(h+1)}$,}\\
	\W^{(h)} & \mbox{ if $p_{j+1}=\U^{(h)}$ or $\L^{(h)}$,}\\
	\D^{(h+1)} & \mbox{ if $p_{j+1}=\D^{(h+1)}$ or $\W^{(h+1)}$,}
	\end{array}
	\right.
\]
for some $h\ge 0$. That $\Phi(\mu)\in\H_{n-1}$ follows from the weight function of the paths in $\M_n$.

It is straightforward to obtain the map $\Phi^{-1}$ by the reverse procedure.
\end{proof}

\begin{exa} {\rm Let $\mu=p_1p_2\dots p_8\in\M_8$ be the path shown on the left-hand side in the upper row of Figure \ref{fig:restructure-path}, where $w_j=\rho(p_j)$ for $1\le j\le 8$. The corresponding  bicolored Motzkin path $\Phi(\mu)=p'_1p'_2\cdots p'_7\in\H_7$  is shown on the right-hand side, where the intermediate stage $z_1z_2\cdots z_{16}$ is shown in the lower row of Figure \ref{fig:restructure-path}.
}
\end{exa}

\begin{figure}[ht]
\begin{center}
\psfrag{1}[][][0.85]{$1$}
\psfrag{w1}[][][0.85]{$w_1$}
\psfrag{w2}[][][0.85]{$w_2$}
\psfrag{w3}[][][0.85]{$w_3$}
\psfrag{w4}[][][0.85]{$w_4$}
\psfrag{w5}[][][0.85]{$w_5$}
\psfrag{w6}[][][0.85]{$w_6$}
\psfrag{w7}[][][0.85]{$w_7$}
\psfrag{w8}[][][0.85]{$w_8$}
\psfrag{Phi}[][][0.85]{$\Phi$}
\includegraphics[width=4.8in]{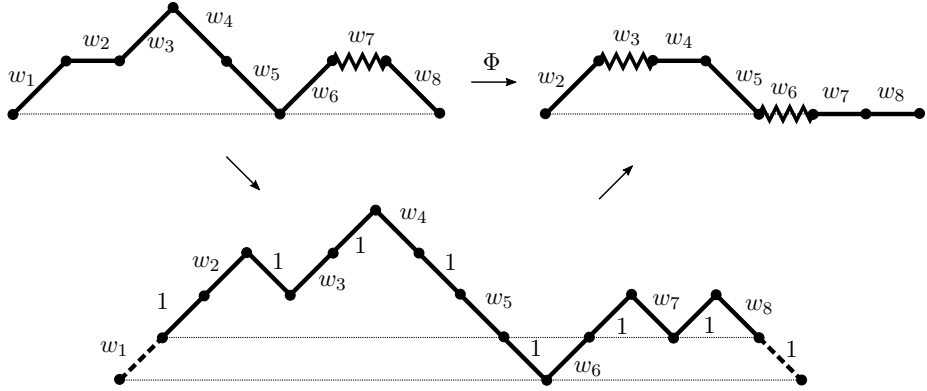}
\end{center}
\caption{\small A restructured weighted bicolored Motzkin paths.}
\label{fig:restructure-path}
\end{figure}

	By a \emph{matching} pair $(\U^{(h)},\D^{(h+1)})$ we mean an up step $\U^{(h)}$ and a down step $\D^{(h+1)}$ that face each other, in the sense that the horizontal line segment from the midpoint of $\U^{(h)}$ to the midpoint of $\D^{(h+1)}$ stays under the path.
	Let $\rho(\U^{(h)},\D^{(h+1)})$ denote the ordered pair $(\rho(\U^{(h)}),\rho(\D^{(h+1)}))$ of weights. We shall establish an involution $\Psi_1:\H_n\rightarrow\H_n$ that changes the weight of a path by the factor $y^2$, with the following set of restricted paths as the fixed points.

\begin{defi} \label{def:Fn} {\rm
Let $\F_n\subset\H_n$ be the subset consisting of the weighted paths satisfying the following conditions. For $h\ge 0$,
\begin{itemize}
\item $\rho(\U^{(h)},\D^{(h+1)})=(y^2q^a,q^b)$ or $(ytq^{h+1+a},ytq^{h+1+b})$ for some $a\in\{0,1,\dots, h+1\}$ and $b\in\{0,1,\dots, h\}$, for any matching pair $(\U^{(h)},\D^{(h+1)})$,
	\item $\rho(\L^{(h)})\in\{ytq^{h+1}, ytq^{h+2},\dots,ytq^{2h+1} \}$,
	\item $\rho(\W^{(h)})\in\{ytq^{h},ytq^{h+1},\dots,ytq^{2h}\}$.
\end{itemize}
}
\end{defi}

Notice that for any matching pair $(\U^{(h)},\D^{(h+1)})$ with weight $(a,b)$, it is equivalent to the reassignment $\rho(\U^{(h)},\D^{(h+1)})=(a',b')$ such that $a'b'=ab$ regarding the total weight of a path.
Comparing the weight functions of the paths in $\F_n$ and in the set $\T_n$ given in Proposition \ref{pro:Tn-weighting}, we have the following result.
	
\begin{lem} \label{lem:Fn=y^nRn} We have
\[
\rho(\F_n)=y^n R_n(t,q).
\]
\end{lem}

\begin{proof} For any path $\mu\in\F_n$, notice that the weight of $\mu$ contains the factor $y^n$ since every matching pair of up step and down step contributes the parameter $y^2$, while every level step (either straight or wavy) contributes the parameter $y$.
By Theorem \ref{thm:Flajolet}, we observe that $\rho(\F_n)$ and $y^n\cdot \rho(\T_n)$ have same generating function. By Proposition \ref{pro:Tn-weighting}, the assertion follows.
\end{proof}

\smallskip
The `sign-reversing' map $\Psi_1:\H_n\rightarrow\H_n$ is constructed as follows.

\smallskip
\noindent
{\bf Algorithm A.}

Given a path $\mu\in\H_n$, the corresponding path $\Psi_1(\mu)$ is constructed by the following procedure.	
\begin{enumerate}
	\item[(A1)] If $\mu$ contains no straight level steps $\L^{(h)}$ with weight $q^a$ or wavy level steps $\W^{(h)}$ with weight $y^2q^a$ for any $a\in\{0,1,\dots,h\}$ then go to (A2). Otherwise, among such level steps find the first step $z$, say $z=\L^{(h)}$ ($\W^{(h)}$, respectively) with weight $\rho(z)=q^a$ ($y^2q^a$, respectively), then the corresponding path $\Psi_1(\mu)$ is obtained by replacing $z$ by $\W^{(h)}$ ($\L^{(h)}$, respectively) with weight $y^2q^a$ ($q^a$, respectively).
	\item[(A2)] If $\mu$ contains no matching pairs $(\U^{(h)},\D^{(h+1)})$ with $\rho(\U^{(h)},\D^{(h+1)})=(y^2q^a, ytq^{h+1+b})$ or $(ytq^{h+1+a}, q^{b})$ for any $a\in\{0,1,\dots, h+1\}$ or  $b\in\{0,1,\dots h\}$ then go to (A3). Otherwise, among such pairs find the first pair with weight, say $(y^2q^a, ytq^{h+1+b})$ ($(ytq^{h+1+a}, q^{b})$, respectively), then the corresponding path $\Psi_1(\mu)$ is obtained by replacing the weight of the pair by $(ytq^{h+1+a}, q^{b})$ ($(y^2q^a, ytq^{h+1+b})$, respectively).
	\item[(A3)] Then we have $\mu\in\F_n$. Let $\Psi_1(\mu)=\mu$.
	\end{enumerate}

\smallskip
Regarding the possibilities of the weighted steps of the paths in $\H_n$, we have the following immediate result.
	
\begin{pro} \label{pro:map-psi1} The map $\Psi_1$ established by Algorithm A is an involution on the set $\H_n$ such that for any path $\mu\in\H_n$,  $\Psi_1(\mu)=\mu$ if $\mu\in\F_n$, and $\rho(\Psi_1(\mu))=y^2\rho(\mu)$ or $y^{-2}\rho(\mu)$ otherwise.
\end{pro}	

Now we are ready to prove (i) of Theorem \ref{thm:signFwexB}.

\smallskip
\noindent
\emph{Proof of (i) of Theorem \ref{thm:signFwexB}.} By Theorem \ref{thm:Corteel} and Proposition \ref{pro:restructure}, we have
\[
\sum_{\sigma\in B_n} y^{\fwex(\sigma)}t^{\nega(\sigma)}q^{\cro_B(\sigma)} =(y^2+yt)\rho(\H_{n-1}).
\]
Then the expression $\sum_{\sigma\in B_n}(-1)^{\lfloor \frac{\fwex(\sigma)}{2}\rfloor} t^{\nega(\sigma)}q^{\cro_B(\sigma)}$ equals the sum of the real part and the imaginary part of the polynomial $(y^2+yt)\rho(\H_{n-1})$ evaluated at $y=\sqrt{-1}$. By  Proposition \ref{pro:map-psi1} and Lemma \ref{lem:Fn=y^nRn}, we have
\begin{align*}
(y^2+yt)\rho(\H_{n-1})\big|_{y=\sqrt{-1}}
            &= (y^2+yt)\rho(\F_{n-1})\big|_{y=\sqrt{-1}}  \\
            &= (y^{n+1}+y^nt) R_{n-1}(t,q)\big|_{y=\sqrt{-1}}.
\end{align*}
Taking the real part and the imaginary part of the above evaluation leads to
\[
\sum_{\sigma\in B_n}(-1)^{\lfloor \frac{\fwex(\sigma)}{2}\rfloor} t^{\nega(\sigma)}q^{\cro_B(\sigma)}=\left[(-1)^{\lfloor\frac{n+1}{2}\rfloor}+(-1)^{\lfloor\frac{n}{2}\rfloor} t \right] R_{n-1}(t,q),
\] 	
as required. \qed

\smallskip
In the following we shall prove (ii) of Theorem \ref{thm:signFwexB}. Recall that the set $D_n$ of even-signed permutations consists of the signed permutations with even number of negative entries.

\begin{defi} {\rm
Let $\M'_n\subset \M_n$ be the subset consisting of the  paths whose weights contain even powers of $t$. Let $\H^{(1)}_n$ ($\H^{(2)}_n$, respectively) be the subset of $\H_n$ consisting of the paths whose weights contain odd (even, respectively) powers of $t$.
}
\end{defi}

Notice that the bijection $\Gamma: B_n\rightarrow\M_n$ in Theorem \ref{thm:Corteel} induces a bijection between $D_n$ and $\M'_n$ such that
\begin{equation} \label{eqn:Dn-M'n}
\sum_{\sigma\in D_n} y^{\fwex(\sigma)}t^{\nega(\sigma)}q^{\cro_B(\sigma)} = \rho(\M'_n).
\end{equation}
Moreover, the involution $\Psi_1:\H_n\rightarrow\H_n$ and the set $\F_n$ of fixed points have the following properties.

\begin{lem} \label{lem:induced-map-psi1} The map $\Psi_1$ established by Algorithm A induces an involution on $\H^{(1)}_n$ and $\H^{(2)}_n$, respectively. Moreover, for any path $\mu\in\F_n$, the power of $t$ of $\rho(\mu)$ has the same parity of $n$.
\end{lem}

\begin{proof} By Proposition \ref{pro:map-psi1}, we observe that the map $\Psi_1:\H_n\rightarrow\H_n$ preserves the powers of $t$ of the weight of the paths. By the
weight conditions of $\mu\in\F_n$ given in Definition \ref{def:Fn}, we observe that every matching pair $(\U^{(h)},\D^{(h+1)})$ contributes the parameter $t^0$ or $t^2$ to $\rho(\mu)$, while every level step contributes the parameter $t$ to $\rho(\mu)$. The assertions follow.
\end{proof}

Now we are ready to prove (ii) of Theorem \ref{thm:signFwexB}.

\smallskip
\noindent
\emph{Proof of (ii) of Theorem \ref{thm:signFwexB}.} By Proposition \ref{pro:restructure} and Eq.\,(\ref{eqn:Dn-M'n}),  taking the terms with even powers of $t$ yields
\[
\sum_{\sigma\in D_n} y^{\fwex(\sigma)}t^{\nega(\sigma)}q^{\cro_B(\sigma)} =y^2\rho(\H^{(2)}_{n-1})+yt\rho(\H^{(1)}_{n-1}).
\]
Consider the polynomial $y^2\rho(\H^{(2)}_{n-1})+yt\rho(\H^{(1)}_{n-1})$ evaluated at $y=\sqrt{-1}$. By Proposition \ref{pro:map-psi1} and Lemma  \ref{lem:induced-map-psi1}, for $n$ odd, we have $\rho(\H^{(1)}_{n-1})\big|_{y=\sqrt{-1}}=0$ and
\begin{align*}
y^2\rho(\H^{(2)}_{n-1})\big|_{y=\sqrt{-1}}
            &= y^2\rho(\F_{n-1})\big|_{y=\sqrt{-1}}  \\
            &= y^{n+1} R_{n-1}(t,q)\big|_{y=\sqrt{-1}},
\end{align*}
Moreover, for $n$ even, we have $\rho(\H^{(2)}_{n-1})\big|_{y=\sqrt{-1}}=0$ and
\begin{align*}
yt\rho(\H^{(1)}_{n-1})\big|_{y=\sqrt{-1}}
            &= yt\rho(\F_{n-1})\big|_{y=\sqrt{-1}}  \\
            &= y^{n}t R_{n-1}(t,q)\big|_{y=\sqrt{-1}}.
\end{align*}
Hence we have
\[
\sum_{\sigma\in D_n}(-1)^{\lfloor \frac{\fwex(\sigma)}{2}\rfloor} t^{\nega(\sigma)}q^{\cro_B(\sigma)}=\left\{
               \begin{array}{ll}
                (-1)^{\frac{n+1}{2}} R_{n-1}(t,q) &\mbox{if $n$ is odd,} \\
                (-1)^{\frac{n}{2}}tR_{n-1}(t,q) &\mbox{if $n$ is even.}
               \end{array}
               \right.
\] 	
The proof of Theorem \ref{thm:signFwexB} is completed.

\section{Proof of Theorem \ref{thm:signFwexD}}
In this section we prove Theorem \ref{thm:signFwexD} in terms of the weighted paths associated to the set $B^{*}_n$ of signed permutations without fixed points.

By the definition of the crossings of signed permutations, for any $\sigma=\sigma_1\sigma_2\cdots\sigma_n\in B_n$ we observe that if $(i,j)$ is a crossing of $\sigma$ then $\sigma_i\neq i$ and $\sigma_j\neq j$, i.e., the fixed points of $\sigma$ are not involved in any crossing of $\sigma$. The following fact is a property of the bijection $\Gamma:B_n\rightarrow\M_n$ given in Theorem \ref{thm:Corteel}.

\begin{lem} \label{lem:no-xing} For a $\sigma=\sigma_1\sigma_2\cdots\sigma_n\in B_n$, let $\Gamma(\sigma)=z_1z_2\cdots z_n\in\M_n$ be the corresponding weighted bicolored Motzkin path. Then for $j\in [n]$, $\sigma_j=j$ if and only if
the step $z_j$ is a straight level step with weight $y^2$.
\end{lem}

Let $\M^{*}_n\subset \M_n$ be the subset consisting of the paths containing no straight level steps with weight $y^2$. It follows from Lemma \ref{lem:no-xing} that the bijection $\Gamma:B_n\rightarrow\M_n$ induces a bijection between $B^{*}_n$ and $\M^{*}_n$ such that
\begin{equation} \label{eqn:expression-B*n}
\sum_{\sigma\in B^{*}_n} y^{\fwex(\sigma)}t^{\nega(\sigma)}q^{\cro_B(\sigma)} =\rho(\M^{*}_n).
\end{equation}

We shall establish an involution $\Psi_2:\M^{*}_n\rightarrow\M^{*}_n$ that changes the weight of a path by the factor $y^2q$, with the following set of restricted paths as the fixed points.

\begin{defi} \label{def:Gn} {\rm
Let $\G_n\subset\M^{*}_n$ be the subset consisting of the paths satisfying the following conditions. For $h\ge 0$,
\begin{itemize}
\item $\rho(\U^{(h)},\D^{(h+1)})=(y^2q^a,q^b)$ or $(ytq^{h+a},ytq^{h+1+b})$ for some $a,b\in\{0,1,\dots, h\}$, for any matching pair $(\U^{(h)},\D^{(h+1)})$,
	\item $\rho(\L^{(h)})\in\{ytq^{h}, ytq^{h+1},\dots,ytq^{2h}\}$,
	\item $\rho(\W^{(h)})\in\{ytq^{h},ytq^{h+1},\dots,ytq^{2h-1}\}$ for $h\ge 1$.
\end{itemize}
}
\end{defi}
	
Comparing the weight functions of the paths in $\G_n$ and in the set $\T^{*}_n$ given in Proposition \ref{pro:T*n-weighting},  the following result can be proved by the same argument as in the proof of Lemma \ref{lem:Fn=y^nRn}.
	
\begin{lem} \label{lem:Gn=y^nQn} We have
\[
\rho(\G_n)=y^n Q_n(t,q).
\]
\end{lem}	
	
We describe the construction of the map $\Psi_2:\M^{*}_n\rightarrow\M^{*}_n$.
	
\smallskip	
\noindent
{\bf Algorithm  B}

Given a path $\mu\in\M^{*}_n$, the corresponding path $\Psi_2(\mu)$ is constructed by the following procedure.
\begin{enumerate}
	\item[(B1)] If $\mu$ contains no straight level steps $\L^{(h)}$ with weight $y^2q^a$ or wavy level steps $\W^{(h)}$ with weight $q^{a-1}$ for any $a\in\{1,2,\dots, h\}$ then go to (B2). Otherwise, among such level steps find the first step $z$, say $z=\L^{(h)}$ ($\W^{(h)}$, respectively) with weight $\rho(z)=y^2q^a$ ($q^{a-1}$, respectively), then the path $\Psi_2(\mu)$ is obtained by replacing $z$ by $\W^{(h)}$ ($\L^{(h)}$, respectively) with weight $q^{a-1}$ ($y^2q^a$, respectively).
	\item[(B2)]  If $\mu$ contains no matching pairs $(\U^{(h)},\D^{(h+1)})$ with $\rho(\U^{(h)},\D^{(h+1)})=(y^2q^a, ytq^{h+1+b})$ or $(ytq^{h+a}, q^{b})$ for any $a,b\in\{0,1,\dots, h\}$ then go to (B3). Otherwise, among such pairs find the first pair with weight, say $(y^2q^a, ytq^{h+1+b})$ ($(ytq^{h+a}, q^{b})$, respectively) then the corresponding path is obtained by replacing the weight of the pair by $(ytq^{h+a}, q^{b})$ ($(y^2q^a, ytq^{h+1+b})$, respectively).
	\item[(B3)] Then we have $\mu\in\G_n$. Let $\Psi_2(\mu)=\mu$.
	\end{enumerate}

\smallskip
By the possible weighted steps of the paths in $\M^{*}_n$, we have the following result.
	
\begin{pro} \label{pro:map-psi2} The map $\Psi_2$ established by Algorithm B is an involution on the set $\M^{*}_n$ such that for any path $\mu\in\M^{*}_n$,  $\Psi_2(\mu)=\mu$ if $\mu\in\G_n$, and $\rho(\Psi_2(\mu))=y^2q\rho(\mu)$ or $y^{-2}q^{-1}\rho(\mu)$ otherwise.
\end{pro}	

Now we are ready to prove Theorem \ref{thm:signFwexD}.

\smallskip
\noindent
\emph{Proof of Theorem \ref{thm:signFwexD}.} (i) To compute the expression $\sum_{\sigma\in B^{*}_n} \big(-\frac{1}{q}\big)^{\lfloor \frac{\fwex(\sigma)}{2}\rfloor} t^{\nega(\sigma)}q^{\cro_B(\sigma)}$, by Eq.\,(\ref{eqn:expression-B*n}) we consider the polynomial $\rho(\M^{*}_n)$ evaluated at $y=\sqrt{\frac{-1}{q}}$. By Proposition \ref{pro:map-psi2} and Lemma \ref{lem:Gn=y^nQn}, we have
\begin{align*}
\rho(\M^{*}_n)\big|_{y=\sqrt{\frac{-1}{q}}}
            &= \rho(\G_n)\big|_{y=\sqrt{\frac{-1}{q}}}  \\
            &= y^{n} Q_{n}(t,q)\big|_{y=\sqrt{\frac{-1}{q}}}.
\end{align*}
Hence we have
\begin{equation} \label{eqn:evaluation-Bn*}
\sum_{\sigma\in B^{*}_n}\left(-\frac{1}{q}\right)^{\lfloor \frac{\fwex(\sigma)}{2}\rfloor} t^{\nega(\sigma)}q^{\cro_B(\sigma)}=
                \left(-\frac{1}{q}\right)^{\lfloor\frac{n}{2}\rfloor} Q_n(t,q).
\end{equation} 	

(ii) Recall that $D^{*}_n\subset B^{*}_n$ is the subset consisting of the fixed point-free signed permutations with even number of negative entries.
Let $\M^{*'}_n\subset \M^{*}_n$ be the subset consisting of the paths whose weights contain even powers of $t$.
Notice that the bijection $\Gamma: B_n\rightarrow\M_n$ also induces a bijection between $D^{*}_n$ and $\M^{*'}_n$ such that
\begin{equation} \label{eqn:D*n-M*'n}
\sum_{\sigma\in D^{*}_n} y^{\fwex(\sigma)}t^{\nega(\sigma)}q^{\cro_B(\sigma)} = \rho(\M^{*'}_n).
\end{equation}

Consider the polynomial $\rho(\M^{*'}_n)$ evaluated at $y=\sqrt{\frac{-1}{q}}$.
By Proposition \ref{pro:map-psi2}, we observe that the involution $\Psi_2:\M^{*}_n\rightarrow\M^{*}_n$ preserves the powers of $t$ of the weight of the paths. Moreover, for the fixed points $\mu\in\G_n$ of the map $\Psi_2$, we observe that the power of $t$ of $\rho(\mu)$ has the same parity of $n$. By the expression in Eq.\,(\ref{eqn:evaluation-Bn*}), we have
\[
\sum_{\sigma\in D^{*}_n}\left(-\frac{1}{q}\right)^{\lfloor \frac{\fwex(\sigma)}{2}\rfloor} t^{\nega(\sigma)}q^{\cro_B(\sigma)}=\left\{
               \begin{array}{ll}
                \big(-\frac{1}{q}\big)^{\frac{n}{2}} Q_n(t,q) &\mbox{if $n$ is even,} \\
                0 &\mbox{if $n$ is odd.}
               \end{array}
               \right.
\] 	
The proof of Theorem \ref{thm:signFwexD} is completed.

\section{Refined enumerators for variants of the snakes}
In this section we shall give a combinatorial interpretation for the polynomials $Q_n(t,q)$ and $R_n(t,q)$ as the enumerators for variants of the snakes in signed permutations. We make use of a non-recursive approach, based on Flajolet's combinatorics of continued fractions, to establish a bijection between the snakes of $\SS^{0}_n$ ($\SS^{00}_{n+1}$, respectively) and the weighted bicolored Motzkin paths of $\T^{*}_n$ given in Proposition \ref{pro:T*n-weighting} ($\T_n$ given in Proposition \ref{pro:Tn-weighting}, respectively). The bijections are constructed in
the spirit of the classical Fran\c con--Viennot bijection \cite{FV} by encoding the sign changes and consecutive up-down patterns of snakes into weighted steps of the Motzkin paths.

Given a permutation $\pi=\pi_1\pi_2\cdots\pi_n\in\mathfrak{S}_n$, Arnol'd \cite{Arnold} devised the following rules to determine signs $\epsilon=\epsilon_1\epsilon_2\cdots\epsilon_n\in\{\pm\}^n$ such that $(\epsilon_1\pi_1,\epsilon_2\pi_2,\dots,\epsilon_n\pi_n)$ is a snake; see also \cite{JV}. For $2\le i\le n-1$,
\begin{itemize}
	\item[(R1)] if $\pi_{i-1}<\pi_i<\pi_{i+1}$ then $\epsilon_i\ne \epsilon_{i+1}$,
	\item[(R2)] if $\pi_{i-1}>\pi_i>\pi_{i+1}$ then $\epsilon_{i-1}\ne \epsilon_{i}$,
	\item[(R3)] if $\pi_{i-1}>\pi_i<\pi_{i+1}$ then $\epsilon_{i-1}= \epsilon_{i+1}$.
\end{itemize}

For a snake $\sigma=\sigma_1\sigma_2\dots\sigma_n\in \SS^{0}_n$ ($\SS^{00}_n$, respectively), we make the convention that $\sigma_0=0$ and $\sigma_{n+1}=(-1)^{n}(n+1)$ ($\sigma_0=\sigma_{n+1}=0$, respectively). Let $|\sigma|$ denote the permutation obtained from $\sigma$ by removing the negative signs of $\sigma$, i.e., $|\sigma|_i=|\sigma_i|$ for $0\le i\le n+1$.

Recall that the number of sign changes of $\sigma$ is given by
$\cs(\sigma):=\#\{i: \sigma_i\sigma_{i+1}<0, 0\le i\le n\}$.
For $1\le j\le n$, let $j=|\sigma|_i$ for some $i\in [n]$, and define $\cs(\sigma,j)$, the number of sign changes recorded by the element $j$, by
\[
\cs(\sigma,j)=\left\{
             \begin{array}{ll}
             0              &\mbox{if $|\sigma|_{i-1}>j<|\sigma|_{i+1}$ and $\sigma_{i-1}\sigma_i>0, \sigma_i\sigma_{i+1}>0$,} \\
             2              &\mbox{if $|\sigma|_{i-1}>j<|\sigma|_{i+1}$ and $\sigma_{i-1}\sigma_i<0, \sigma_i\sigma_{i+1}<0$,} \\
             1              &\mbox{if $|\sigma|_{i-1}<j<|\sigma|_{i+1}$ or $|\sigma|_{i-1}>j>|\sigma|_{i+1}$,} \\
             0              &\mbox{if $|\sigma|_{i-1}<j>|\sigma|_{i+1}$.}
             \end{array}
       \right.
\]
We call the sequence $(\cs(\sigma,1),\dots,\cs(\sigma,n))$ the $\cs$-\emph{vector} of $\sigma$.

Following the rules (R1)-(R3) and the condition $\sigma_1>0$, we observe that the signs of the entries $\sigma_1,\dots,\sigma_n$ of $\sigma$ can be recovered from left to right by $|\sigma|$ and the $\cs$-vector of $\sigma$.

\begin{lem} \label{lem:sign-changes} For any snake $\sigma\in\SS^{0}_n$ ($\SS^{00}_n$, respectively), the sign of each entry of $\sigma$ is uniquely determined by $|\sigma|$ and the vector $(\cs(\sigma,1), \dots, \cs(\sigma,n))$. Moreover, we have
\[
\cs(\sigma)=\sum_{j=1}^{n} \cs(\sigma,j).
\]
\end{lem}

\begin{proof} For the initial condition, we have $\sigma_0=0$ and $\sigma_1>0$. For $i\ge 2$, we determine the sign of $\sigma_i$ according to the following cases:

Case 1. $|\sigma|_{i-1}>|\sigma|_i$. There are two cases. If $|\sigma|_i>|\sigma|_{i+1}$ then by (R2) $\sigma_i$ has the opposite sign of $\sigma_{i-1}$.
Otherwise, $|\sigma|_i<|\sigma|_{i+1}$, and by (R3) $\sigma_i$ has the opposite (same, respectively) sign of $\sigma_{i-1}$ if $\cs(|\sigma|,|\sigma|_i)=2$ (0, respectively). Hence the sign change between $\sigma_{i-1}$ and $\sigma_i$ is recorded by $\cs(|\sigma|,|\sigma|_i)$.

Case 2. $|\sigma|_{i-1}<|\sigma|_i$. There are two cases. If $|\sigma|_{i-2}<|\sigma|_{i-1}$ then by (R1) $\sigma_i$ has the opposite sign of $\sigma_{i-1}$.
Otherwise, $|\sigma|_{i-2}>|\sigma|_{i-1}$, and by (R3) $\sigma_i$ has the opposite (same, respectively) sign of $\sigma_{i-1}$ if $\cs(|\sigma|,|\sigma|_{i-1})=2$ (0, respectively). Hence the sign change between $\sigma_{i-1}$ and $\sigma_i$ is recorded by $\cs(|\sigma|,|\sigma|_{i-1})$.

The assertions follow.
\end{proof}

\smallskip
\begin{exa} \label{exa:sign-snake} {\rm Given a snake $\sigma=((0),5,-2,4,-7,-1,-8,10,-9,6,3,(11))\in\SS^{0}_{10}$, note that $\cs(\sigma)=6$ and the $\cs$-vector of $\sigma$ is $(0,2,0,1,0,1,0,1,1,0)$.

On the other hand, given the permutation $|\sigma|=(5,2,4,7,1,8,10,9,6,3)\in\mathfrak{S}_{10}$ and the $\cs$-vector $(0,2,0,1,0,1,0,1,1,0)$, we observe that the snake $\sigma$ can be recovered, following the rules (R1)-(R3).
}
\end{exa}

In the following we shall encode the permutation $|\sigma|$ by a weighted bicolored Motzkin path.
Let $\mathfrak{S}^{0}_n$ and $\mathfrak{S}^{00}_n$ denote two `copies' of $\mathfrak{S}_n$ with the following convention
\begin{itemize}
\item $\pi_0=0$ and $\pi_{n+1}=n+1$ if $\pi\in\mathfrak{S}^{0}_n$,
\item $\pi_0=\pi_{n+1}=0$ if $\pi\in\mathfrak{S}^{00}_n$.
\end{itemize}

Given a permutation $\pi=\pi_1\pi_2\cdots\pi_n\in\mathfrak{S}^{0}_n$ or $\mathfrak{S}^{00}_n$, by a \emph{block} of $\pi$ restricted to $\{0,1,\dots,k\}$ we mean a maximal sequence of consecutive entries $\pi_i\pi_{i+1}\cdots\pi_j\subseteq\{0,1,\dots,k\}$ for some $i\le j$. For $0\le k\le n$, let $\alpha(\pi,k)$ be the number of blocks of $\pi$ restricted to $\{0,1\dots,k\}$, and let $\beta(\pi,k)$ be the number of such blocks on the right-hand side of the block containing the element $k$. By the convention on $\sigma_0$ and $\sigma_{n+1}$, for $k=0$ we have $\alpha(\pi,0)=1$ if $\pi\in\mathfrak{S}^{0}_n$, while $\alpha(\pi,0)=2$ if $\pi\in\mathfrak{S}^{00}_n$.

\begin{exa} \label{exa:alpha-beta-0} {\rm
Let $\pi=((0),5,2,4,7,1,8,10,9,6,3,(11))\in\mathfrak{S}^{0}_{10}$. Notice that $\alpha(\pi,6)=3$ and $\beta(\pi,6)=0$. The three blocks of $\pi$ restricted to $\{0,1,\dots,6\}$ are underlined as shown below.
%\begin{table}[ht]
%\caption{The blocks of $\pi$ restricted to $\{0,1,\dots,6\}$.}
%\begin{center}
\[
\begin{tabular}{cccccccccccc}
    (0) & 5 & 2 & 4 & 7 & 1 & 8 & 10 & 9 & 6 & 3 & (11) \\
    \cline{1-4} \cline{6-6}  \cline{10-11} \\

\end{tabular}
\]
%\end{center}
%\label{tab:blocks-0}
%\end{table}
For $0\le k\le 10$, the sequences of $\alpha(\pi,k)$ and $\beta(\pi,k)$ of $\pi$ are shown in Table \ref{tab:block-vector}.
}
\end{exa}

\begin{table}[ht]
\caption{The sequences $\alpha,\beta$ of $\pi=((0),5,2,4,7,1,8,10,9,6,3,(11))$.}
\begin{tabular}{c|ccccccccccc}
    $k$           & 0 & 1 & 2 & 3 & 4 & 5 & 6 & 7 & 8 & 9 & 10 \\
    \hline
  $\alpha(\pi,k)$ & 1 & 2 & 3 & 4 & 4 & 3 & 3 & 2 & 2 & 2 & 1 \\
   $\beta(\pi,k)$ & 0 & 0 & 1 & 0 & 2 & 2 & 0 & 1 & 1 & 0 & 0
\end{tabular}
\label{tab:block-vector}
\end{table}

\subsection{The enumerator $Q_n(t,q)$ of $\SS^{0}_n$.}
We shall establish a map $\Lambda_1:\SS^{0}_n\rightarrow\T^{*}_n$ by the following procedure.

\smallskip
\noindent
{\bf Algorithm C.}

Given a snake $\sigma=\sigma_1\sigma_2\cdots\sigma_n\in\SS^{0}_n$, we associate $\sigma$ with a weighted path $\Lambda_1(\sigma)=z_1z_2\cdots z_n$.
For $1\le j\le n$, let $j=|\sigma|_i$ for some $i\in [n]$ and define the step $z_j$ according to the following cases:
\begin{enumerate}
\item if $|\sigma|_{i-1}>j<|\sigma|_{i+1}$ then $z_j=\U$ with weight
\[
\rho(z_j)=\left\{ \begin{array}{ll}
                   q^{\beta(|\sigma|,j)} &\mbox{ if $\sigma_{i-1}\sigma_i>0$ and $\sigma_i\sigma_{i+1}>0$,} \\
                   t^2q^{\beta(|\sigma|,j)+2\alpha(|\sigma|,j)-3} &\mbox{ if $\sigma_{i-1}\sigma_i<0$ and $\sigma_i\sigma_{i+1}<0$,}
                  \end{array}
          \right.
\]
\item if $|\sigma|_{i-1}<j<|\sigma|_{i+1}$ then $z_j=\L$ with weight $tq^{\beta(|\sigma|,j)+\alpha({|\sigma|,j})-1}$,
\item if $|\sigma|_{i-1}>j>|\sigma|_{i+1}$ then $z_j=\W$ with weight $tq^{\beta(|\sigma|,j)+\alpha({|\sigma|,j})-1}$,
\item if $|\sigma|_{i-1}<j>|\sigma|_{i+1}$ then $z_j=\D$ with weight $q^{\beta(|\sigma|,j)}$.
\end{enumerate}

\smallskip
Notice that the value $\cs(\sigma,j)$ is encoded as the power of $t$ in $\rho(z_j)$. For convenience, the element $j$ in (i) is called a \emph{valley}, in (ii) a \emph{double ascent}, in (iii) a \emph{double descent}, and in (iv) a \emph{peak} of $\sigma$.

\smallskip
\begin{exa} \label{exa:snake-0} {\rm
Take the snake $\sigma=((0),5,-2,4,-7,-1,-8,10,-9,6,3,(11))\in\SS^{0}_{10}$. The $\cs$-vector of $\sigma$ is given in Example \ref{exa:sign-snake} and the sequences  $\alpha$, $\beta$ of $|\sigma|$ are given in Example \ref{exa:alpha-beta-0}. We observe that $z_1=\U$ with weight $q^{\beta(|\sigma|,1)}=1$ since the element 1 is a valley without sign-changes, and that $z_2=\U$ with weight $t^2q^{\beta(|\sigma|,2)+2\alpha(|\sigma|,2)-3}=t^2q^4$ since the element 2 is a valley with sign-changes. The path $\Lambda_1(\sigma)$ is shown in Figure \ref{fig:snake-path}.
}
\end{exa}

\begin{figure}[ht]
\begin{center}
\psfrag{1}[][][0.95]{$1$}
\psfrag{t2q4}[][][0.95]{$t^2q^4$}
\psfrag{tq5}[][][0.95]{$tq^5$}
\psfrag{q2}[][][0.95]{$q^2$}
\psfrag{tq2}[][][0.95]{$tq^2$}
\psfrag{q}[][][0.95]{$q$}
\psfrag{tq}[][][0.95]{$tq$}
\includegraphics[width=3.2in]{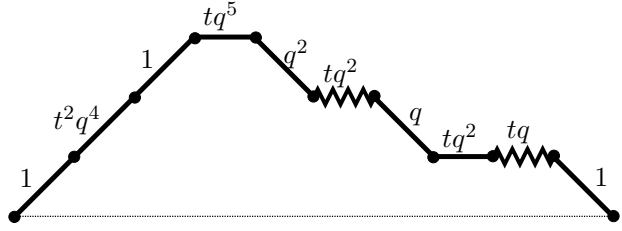}
\end{center}
\caption{\small The corresponding path of the snake in Example \ref{exa:snake-0}.}
\label{fig:snake-path}
\end{figure}

The constraints for the parameters $\alpha(|\sigma|,k)$ and $\beta(|\sigma|,k)$ of $|\sigma|$ are encoded in the height of the step $z_k\in\Lambda_1(\sigma)$.

\begin{lem} \label{lem:heights} For a snake $\sigma=\sigma_1\sigma_2\cdots\sigma_n\in\SS^{0}_n$, let $\Lambda_1(\sigma)=z_1z_2\cdots z_n$ be the path constructed by Algorithm C. For $1\le j\le n$, let $h_j$ be the height of the step $z_j$.  Then the following properties hold.
\begin{enumerate}
\item $h_j=\alpha(|\sigma|,j-1)-1$.
\item If $z_j=\W$ or $\D$ then $h_j\ge 1$ and $0\le \beta(|\sigma|,j)\le h_j-1$.
\item If $z_j=\U$ or $\L$ then $0\le \beta(|\sigma|,j)\le h_j$.
\end{enumerate}
\end{lem}

\begin{proof} For the initial condition, we have $\alpha(|\sigma|,0)=1$ and $h_1=0$. The first step is either $\U$ or $\L$ since the element 1 is either a valley or a double ascent. For $j\ge 1$, let $j=|\sigma|_i$ for some $i\in [n]$. By induction, we determine the height of $z_{j+1}$ according to the following cases:
\begin{itemize}
\item If $|\sigma|_{i-1}>j<|\sigma|_{i+1}$ then $z_j=\U$ and the element $j$ itself creates a block of $|\sigma|$ restricted to $\{0,1,\dots,j\}$. Hence $h_{j+1}=h_{j}+1=\alpha(|\sigma|,j-1)=\alpha(|\sigma|,j)-1$.
\item If $|\sigma|_{i-1}<j<|\sigma|_{i+1}$ or $|\sigma|_{i-1}>j>|\sigma|_{i+1}$ then $z_j=\L$  ($\W$, respectively) and the element $j$ is added to the block with $|\sigma|_{i-1}$ ($|\sigma|_{i+1}$, respectively). Hence $h_{j+1}=h_{j}=\alpha(|\sigma|,j-1)-1=\alpha(|\sigma|,j)-1$.
\item If $|\sigma|_{i-1}<j>|\sigma|_{i+1}$ then $z_j=\D$ and the element $j$ connects the adjacent two blocks. Hence $h_{j+1}=h_{j}-1=\alpha(|\sigma|,j-1)-2=\alpha(|\sigma|,j)-1$.
\end{itemize}
The assertion (i) follows.

(ii) If $z_j=\W$ or $\D$ then $j>|\sigma|_{i+1}$. The element $j$ is added to the block with $|\sigma|_{i+1}$, which is different from the block containing $(0)$. Then $\alpha(|\sigma|,j-1)\ge 2$ and hence $h_j\ge 1$. Moreover, there are at most $\alpha(|\sigma|,j-1)-2$ blocks on the right-hand side of the block containing $|\sigma|_{i+1}$. Hence $\beta(|\sigma|,j)\le h_j-1$.

(iii) If $z_j=\U$ or $\L$ then $j<|\sigma|_{i+1}$ and there are at most $\alpha(|\sigma|,j-1)-1$ blocks on the right-hand side of the block containing $j$. Hence $\beta(|\sigma|,j)\le h_j$.
\end{proof}

Comparing the weight function of the paths in $\T^{*}_n$ in Proposition \ref{pro:T*n-weighting} and the properties of $\Lambda_1(\sigma)$ in  Lemma \ref{lem:heights}, it follows that the path $\Lambda_1(\sigma)$ constructed by Algorithm C is a member of $\T^{*}_n$.

\smallskip
Next, we shall construct the map $\Lambda_1^{-1}:\T^{*}_n\rightarrow\SS^{0}_n$ by the following procedure.

\smallskip
\noindent
{\bf Algorithm D.}

Given a path $\mu=z_1z_2\cdots z_n\in\T^{*}_n$, we associate $\mu$ with a snake $\sigma'=\Lambda_1^{-1}(\mu)$.
For $1\le j\le n$, let $\cs(|\sigma'|,j)$ ($d_j$, respectively) be the power of $t$ ($q$, respectively) of $\rho(z_j)$, and let $h_j$ be the height of  $z_j$. To find $|\sigma'|$,
 we construct a sequence $\omega_0, \omega_1, \dots, \omega_n=|\sigma'|$ of words, where $\omega_j$ is the subword consisting of the blocks of $|\sigma'|$ restricted to $\{0,1,\dots,j\}$. The initial word $\omega_0$ is a singleton $(\sigma_0)$. For $j\ge 1$, the
word $\omega_j$ is constructed from $\omega_{j-1}$ and $\rho(z_j)$ according to the following cases:
\begin{enumerate}
\item $z_j=\U$. There are two cases. If $\cs(|\sigma'|,j)=0$,  let $\ell=d_j$. Otherwise $\cs(|\sigma'|,j)=2$ and let $\ell=d_j-2h_j-1$. Then the word $\omega_j$ is obtained from $\omega_{j-1}$ by inserting $j$ between the $\ell$th and the $(\ell+1)$st block from right as a new block.
\item $z_j=\L$ or $\W$. Then let $\ell=d_j-h_j$. The word $\omega_j$ is obtained from $\omega_{j-1}$ by appending $j$ to the right end (left end, respectively) of the $(\ell+1)$st block from right as a new member of the block if $z_j=\L$ ($\W$, respectively).
\item $z_j=\D$. Then let $\ell=d_j$. The word $\omega_j$ is obtained from $\omega_{j-1}$ by inserting $j$ between the $(\ell+1)$st and the $(\ell+2)$nd block from right and getting the two blocks combined.
\end{enumerate}
Then following the rules (R1)-(R3), we determine the signs of the elements of $\omega_n$ by the sequence $\cs(|\sigma'|,j)$ for $j=1,2,\dots, n$. Hence the requested snake $\sigma'$ is established.

\smallskip

In the following we give an interpretation of the sequences $\alpha$, $\beta$ in terms of three-term patterns of the permutation $|\sigma|$.

\begin{defi} {\rm
Let $\pi=\pi_1\pi_2\cdots\pi_n\in\mathfrak{S}^{0}_n$ or $\mathfrak{S}^{00}_n$. For $1\le i\le n$, we define
\begin{align*}
\acb(\pi,i) &= \#\{j: 0\le j<i-1 \mbox{ and }\pi_j<\pi_i<\pi_{j+1}\}, \\
\bca(\pi,i) &= \#\{j: i< j\le n \mbox{ and }\pi_j>\pi_i>\pi_{j+1}\}.
\end{align*}
Let also $\bca(\pi)=\sum_{i=1}^{n} \bca(\pi,i)$. For any snake $\sigma=\sigma_1\sigma_2\cdots\sigma_n\in\SS^{0}_n$ or $\SS^{00}_n$, we distinguish the following classes $X$, $Y$ and $Z$ of elements of $\sigma$:  (i) the valleys with sign changes,   (ii) the double ascents or double descents and (iii) the peaks, namely
\begin{align*}
X(\sigma) &=\{|\sigma|_i : |\sigma|_{i-1}>|\sigma|_i<|\sigma|_{i+1},  \sigma_{i-1}\sigma_i<0 \mbox{ and } \sigma_i\sigma_{i+1}<0\}, \\
Y(\sigma) &=\{|\sigma|_i : |\sigma|_{i-1}<|\sigma|_i<|\sigma|_{i+1} \mbox{ or } |\sigma|_{i-1}>|\sigma|_i>|\sigma|_{i+1}\}, \\
Z(\sigma) &=\{|\sigma|_i : |\sigma|_{i-1}<|\sigma|_i>|\sigma|_{i+1}\}.
\end{align*}
}
\end{defi}

The parameters $\alpha(|\sigma|,k)$ and $\beta(|\sigma|,k)$ of the snake $\sigma\in\SS^{0}_n$ or $\SS^{00}_{n}$ have the following properties.

\begin{lem} \label{lem:pattern} For $0\le k\le n$, we have
\begin{enumerate}
\item $\beta(|\sigma|,k)=\bca(|\sigma|,k)$,
\item $\alpha(|\sigma|,k)=\acb(|\sigma|,k)+\bca(|\sigma|,k)+1$.
\end{enumerate}
\end{lem}

\begin{proof} Suppose there are $\ell$ ($\ell'$, respectively) blocks on the right-hand (left-hand, respectively) side of the block containing $k$ when $|\sigma|$ is restricted to $\{0,1,\dots,k\}$. Then along with the element $k$, the two adjacent entries of $|\sigma|$ at the left (right, respectively) boundary of each block constitute a $\bca$-pattern ($\acb$-pattern, respectively). Hence $\bca(|\sigma|,k)=\ell$ and $\acb(|\sigma|,k)=\ell'$. The assertions follow.
\end{proof}

% \begin{exa} \label{exa:13-2-pattern} {\rm
For example, let $\pi=((0),5,2,4,7,1,8,10,9,6,3,(11))\in\mathfrak{S}^{0}_{10}$. As shown in Example \ref{exa:alpha-beta-0}, $\alpha(\pi,6)=3$ and $\beta(\pi,6)=0$.
We have $\acb(\pi,6)=2$, where the two requested $\acb$-patterns are $(4,7,6)$ and $(1,8,6)$.
% }
% \end{exa}

Following the weighting scheme given in Algorithm C,
we define the statistic $\pat_Q$ of a snake $\sigma\in\SS^{0}_n$ by
\begin{align} \label{eqn:patQ}
\begin{split}
\pat_Q(\sigma) &=\sum_{j\in X(\sigma)} 2\big(\acb(|\sigma|,j)+\bca(|\sigma|,j)\big)-\#X(\sigma) \\
               &\qquad\qquad +\sum_{j\in Y(\sigma)} \big( \acb(|\sigma|,j)+\bca(|\sigma|,j)\big).
\end{split}
\end{align}
By Lemmas \ref{lem:heights} and \ref{lem:pattern} and Proposition \ref{pro:T*n-weighting},
we have the following result.

\smallskip
\begin{thm} \label{thm:T*->snake-0} The map $\Lambda_1$ established by Algorithm C is a bijection between $\SS^{0}_n$ and $\T^{*}_n$ such that
\[
\sum_{\sigma\in \SS_n^{0}} t^{\cs(\sigma)}q^{\bca(|\sigma|)+\pat_Q(\sigma)}=Q_n(t,q).
\]
\end{thm}

\subsection{The enumerator $R_n(t,q)$ of $\SS^{00}_{n+1}$.} We shall establish a map $\Lambda_2:\SS^{00}_{n+1}\rightarrow\T_n$ by the same method as in Algorithm C with a modification on the weighting scheme.

Given a snake $\sigma=\sigma_1\sigma_2\cdots\sigma_{n+1}\in\SS^{00}_{n+1}$, recall that the $\cs$-vector of $\sigma$ and the parameters $\alpha(|\sigma|,k)$, $\beta(|\sigma|,k)$ for $k=1,2,\dots, n$ are computed under the convention $\sigma_0=\sigma_{n+2}=0$.

\begin{exa} \label{exa:alpha-beta-00} {\rm
Let $\sigma=((0),5,-2,4,-7,-1,-8,11,-9,6,3,10,(0))\in\SS^{00}_{11}$. Notice that $\alpha(|\sigma|,6)=4$ and $\beta(|\sigma|,6)=1$ as shown below.
% \begin{table}[ht]
% \caption{The blocks of $|\sigma|$ restricted to $\{0,1,\dots,6\}$.}
% \begin{center}
\[
\begin{tabular}{ccccccccccccc}
    (0) & 5 & 2 & 4 & 7 & 1 & 8 & 11 & 9 & 6 & 3 & 10 & (0) \\
    \cline{1-4} \cline{6-6}  \cline{10-11} \cline{13-13} \\
\end{tabular}
\]
% \end{center}
% \label{tab:block-00}
% \end{table}
For $0\le k\le 10$, the sequences $\alpha(|\sigma|,k)$ and $\beta(|\sigma|,k)$ of $|\sigma|$ are shown in Table \ref{tab:block-vector-00}.
}
\end{exa}

\begin{table}[ht]
\caption{The $\alpha$ and $\beta$ vectors of $|\sigma|$.}
\begin{tabular}{c|cccccccccccc}
    $k$           & 0 & 1 & 2 & 3 & 4 & 5 & 6 & 7 & 8 & 9 & 10 & 11 \\
    \hline
  $\alpha(|\sigma|,k)$ & 2 & 3 & 4 & 5 & 5 & 4 & 4 & 3 & 3 & 3 & 2  &    \\
   $\beta(|\sigma|,k)$ &   & 1 & 2 & 1 & 3 & 3 & 1 & 2 & 2 & 1 & 0  &
\end{tabular}
\label{tab:block-vector-00}
\end{table}

We associate the snake $\sigma$ with a weighted path $\Lambda_2(\sigma)=z_1z_2\cdots z_n$ by the following procedure.

\smallskip
\noindent
{\bf Algorithm C'.}

For $1\le j\le n$, let $j=|\sigma|_i$ for some $i\in [n+1]$ and define the step $z_j$ according to the following cases:
\begin{enumerate}
\item if $|\sigma|_{i-1}>j<|\sigma|_{i+1}$ then $z_j=\U$ with weight
\[
\rho(z_j)=\left\{ \begin{array}{ll}
                   q^{\beta(|\sigma|,j)-1} &\mbox{ if $\sigma_{i-1}\sigma_i>0$ and $\sigma_i\sigma_{i+1}>0$,} \\
                   t^2q^{\beta(|\sigma|,j)+2\alpha(|\sigma|,j)-5} &\mbox{ if $\sigma_{i-1}\sigma_i<0$ and $\sigma_i\sigma_{i+1}<0$,}
                  \end{array}
          \right.
\]
\item if $|\sigma|_{i-1}<j<|\sigma|_{i+1}$ then $z_j=\L$ with weight $tq^{\beta(|\sigma|,j)+\alpha({|\sigma|,j})-2}$,
\item if $|\sigma|_{i-1}>j>|\sigma|_{i+1}$ then $z_j=\W$ with weight $tq^{\beta(|\sigma|,j)+\alpha({|\sigma|,j})-2}$,
\item if $|\sigma|_{i-1}<j>|\sigma|_{i+1}$ then $z_j=\D$ with weight $q^{\beta(|\sigma|,j)}$.
\end{enumerate}

\smallskip
For example, take the snake $\sigma=((0),5,-2,4,-7,-1,-8,11,-9,6,3,10,(0))\in\SS^{00}_{11}$. From the parameters $\alpha(|\sigma|,k)$, $\beta(|\sigma|,k)$ of $\sigma$ given in Example \ref{exa:alpha-beta-00}, the corresponding path $\Lambda_2(\sigma)$ is shown in Figure \ref{fig:snake-path-00}.

\begin{figure}[ht]
\begin{center}
\psfrag{1}[][][0.95]{$1$}
\psfrag{t2q5}[][][0.95]{$t^2q^5$}
\psfrag{tq6}[][][0.95]{$tq^6$}
\psfrag{q3}[][][0.95]{$q^3$}
\psfrag{tq3}[][][0.95]{$tq^3$}
\psfrag{q2}[][][0.95]{$q^2$}
\psfrag{tq2}[][][0.95]{$tq^2$}
\includegraphics[width=3.2in]{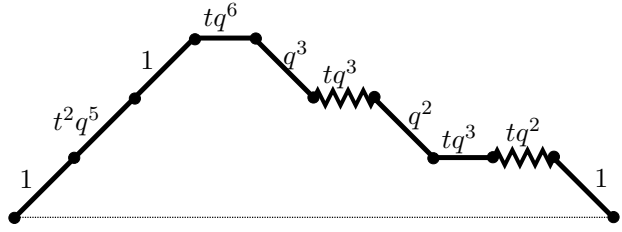}
\end{center}
\caption{\small The corresponding path of the snake $\sigma\in\SS^{00}_{11}$.}
\label{fig:snake-path-00}
\end{figure}

\begin{lem} \label{lem:heights-00} For a snake $\sigma=\sigma_1\sigma_2\cdots\sigma_{n+1}\in\SS^{00}_{n+1}$, let $\Lambda_2(\sigma)=z_1z_2\cdots z_n$ be the path constructed by Algorithm C'. For $1\le j\le n$, let $h_j$ be the height of the step $z_j$. Then the following properties hold.
\begin{enumerate}
\item $h_j=\alpha(|\sigma|,j-1)-2$.
\item If $z_j=\D$ then $h_j\ge 1$.
\item If $z_j=\W$ or $\D$ then $0\le \beta(|\sigma|,j)\le h_j$.
\item If $z_j=\U$ or $\L$ then $1\le \beta(|\sigma|,j)\le h_j+1$.
\end{enumerate}
\end{lem}

\begin{proof}
By the initial condition $\alpha(|\sigma|,0)=2$, we have $h_1=0$. The first step is $\U$, $\L$ or $\W$. Note that $z_1=\L$ ($\W$, respectively) if $|\sigma|_1=1$ ($|\sigma|_{n+1}=1$, respectively) and $z_1=\U$ if $|\sigma|_i=1$ for some $i\in\{2,\dots,n\}$. For $j\ge 1$, let $j=|\sigma|_i$ for some $i\in [n+1]$.
The assertion (i) can be proved by the same argument as in the proof of (i) of Lemma \ref{lem:heights}.

(ii) Note that $\alpha(|\sigma|,j)\ge 2$ since the greatest element $n+1$ is always absent. If $z_j=\D$ then $\alpha(|\sigma|,j-1)\ge 3$ and hence $h_j\ge 1$.

(iii) If $z_j=\W$ or $\D$ then $j>|\sigma|_{i+1}$ and the element $j$ is added to the block with $|\sigma|_{i+1}$. Then there are at most $\alpha(|\sigma|,j-1)-2$ blocks on the right-hand side of the block containing $|\sigma|_{i+1}$. Hence $0\le\beta(|\sigma|,j)\le h_j$.

(iv) If $z_j=\U$ or $\L$ then $j<|\sigma|_{i+1}$ and there are at least one block and at most $\alpha(|\sigma|,j-1)-1$ blocks on the right-hand side of the element $j$. Hence $1\le\beta(|\sigma|,j)\le h_j+1$.
\end{proof}

\smallskip
Comparing the weight function of the paths in $\T_n$ in Proposition \ref{pro:Tn-weighting} and the properties of $\Lambda_2(\sigma)$ in  Lemma \ref{lem:heights-00}, it follows that the path $\Lambda_2(\sigma)$ constructed by Algorithm C' is a member of $\T_n$.

\smallskip
To find $\Lambda^{-1}_2$, given a path $\mu=z_1z_2\cdots z_n\in\T_n$, we shall construct the corresponding snake $\sigma'=\Lambda_2^{-1}(\mu)$ by the following procedure.

\smallskip
\noindent
{\bf Algorithm D'.}

For $1\le j\le n$, let $\cs(|\sigma'|,j)$ ($d_j$, respectively) be the power of $t$ ($q$, respectively) of $\rho(z_j)$ and let $h_j$ be the height of $z_j$. To find $|\sigma'|$,
 we construct a sequence $\omega_0,\omega_1,\dots,\omega_n$ of words, where $\omega_j$ is the subword consisting of the blocks of $|\sigma'|$ restricted to $\{0,1,\dots,j\}$. The last word $\omega_n$ contains exactly two blocks, and the requested permutation $|\sigma'|$ is obtained from $\omega_n$ by inserting the element $n+1$ between the two blocks.

The initial word $\omega_0$ consists of the two blocks $(\sigma_0)$ and $(\sigma_{n+1})$. For $j\ge 1$, the
word $\omega_j$ is constructed from $\omega_{j-1}$ and $\rho(z_j)$ according to the following cases:
\begin{enumerate}
\item $z_j=\U$. There are two cases. If $\cs(|\sigma'|,j)=0$,  let $\ell=d_j+1$. Otherwise $\cs(|\sigma'|,j)=2$ and let $\ell=d_j-2h_j-1$. Then the word $\omega_j$ is obtained from $\omega_{j-1}$ by inserting $j$  between the $\ell$th and the $(\ell+1)$st block from right as a new block.
\item $z_j=\L$ or $\W$. Then let $\ell=d_j-h_j$. The word $\omega_j$ is obtained from $\omega_{j-1}$ by appending $j$ to the right end (left end, respectively) of the $(\ell+1)$st block from right as a new member of the block if $z_j=\L$ ($\W$, respectively).

\item $z_j=\D$. Then let $\ell=d_j$. The word $\omega_j$ is obtained from $\omega_{j-1}$ by inserting $j$ to $\omega_{j-1}$ between the $(\ell+1)$st and the $(\ell+2)$nd block from right and getting the two blocks combined.
\end{enumerate}
Then the signs of the elements of the requested snake $\sigma'$ can be determined by $|\sigma'|$ and the sequence $\cs(|\sigma'|,j)$ for $j=1,2,\dots,n$.

\smallskip
Following the weighting scheme given in Algorithm C', we define the statistic $\pat_R$ of a snake $\sigma\in\SS^{00}_{n+1}$ by
\begin{align} \label{eqn:patR}
\begin{split}
\pat_R(\sigma) &=\sum_{j\in X(\sigma)} 2\big(\acb(|\sigma|,j)+\bca(|\sigma|,j)-1\big) \\
               &\qquad\qquad +\sum_{j\in Y(\sigma)} \big(\acb(|\sigma|,j)+\bca(|\sigma|,j)\big)+\#Z(\sigma).
\end{split}
\end{align}
Notice that $\sum_{j=1}^{n} (\beta(|\sigma|,j)-1)=\bca(|\sigma|)-n$ and that $Z(\sigma)$ contains the element $n+1$, which is not involved in step (iv) of Algorithm C'.
By Lemmas \ref{lem:heights-00} and \ref{lem:pattern} and Proposition \ref{pro:Tn-weighting},
we have the following result.

\smallskip
\begin{thm} \label{thm:T->snakes-00} The map $\Lambda_2$ established by Algorithm C' is a bijection between $\SS^{00}_{n+1}$ and $\T_n$ such that
\[
\sum_{\sigma\in \SS^{00}_{n+1}} t^{\cs(\sigma)}q^{\bca(|\sigma|)+\pat_R(\sigma)-n-1}=R_n(t,q).
\]
\end{thm}

\section{Concluding remarks}

The first part of this work gives signed countings of types B and D permutations and derangements with the signs controlled by one half of the statistic $\fwex$. The type B case happens to be the Springer numbers of type B. However, the type D case is \emph{not} the Springer numbers of type D. Hence a natural question is to find a statistic on even-signed permutations that  results in the Springer numbers of type D. The $\gamma$-nonnegativity result in \cite[Theorem 13.9]{Petersen} sheds some light on the question.

The second part of this work gives an interpretation of the $t,q$-Euler numbers $Q_n(t,q)$ and $R_n(t,q)$. However the statistics $\pat_Q$ and $\pat_R$ for the variate $q$ are complicated. We would like to know if there are more `natural' interpretations. These questions are left to the interested readers.

\section*{Acknowledgements.}
The authors thank the referees for reading the manuscript carefully and providing helpful suggestions.
This research is partially supported by MOST grants 104-2115-M-003-014-MY3 (S.-P. Eu), 105-2115-M-153-002-MY2 (T.-S. Fu) and MOST postdoctoral fellowship 106-2811-M-003-011 (H.-C. Hsu).

\end{document}